\documentclass[oneside]{amsart}
\usepackage{amssymb,amsmath}
\usepackage{graphicx}
\usepackage{epstopdf}

\def \N{\mathbb{N}}

\newtheorem{theorem}{Theorem}[section]

\newtheorem{corollary}[theorem]{Corollary}

\begin{document}

\title[Reliability of $d$-dimensional consecutive systems]{A formula for the reliability of a $d$-dimensional consecutive-$k$-out-of-$n$:F system}
\author{Simon Cowell}
\address{Department of Mathematical Sciences, UAE University, Al Ain, Abu Dhabi, United Arab Emirates}
\email{scowell@uaeu.ac.ae}
\date{17th of March 2015}
\begin{abstract}
We derive a formula for the reliability of a $d$-dimensional consecutive-$k$-out-of-$n$:F system. That is, a formula for the probability that an $n_1 \times \ldots \times n_d$ array whose entries are (independently of each other) 0 with probability $p$ and 1 with probability $q=1-p$ does not include a contiguous $s_1 \times \ldots \times s_d$ subarray whose every entry is 1.
\end{abstract}
\keywords{consecutive-k-out-of-n:F system, reliability, probability, binary matrices, integer sequences}
\maketitle


\section{Introduction}

Consider a pipeline punctuated by several pumping stations. Suppose the pumping stations are redundant in that, if one of them fails, then its predecessor will be strong enough to pump the fluid past it to the next pumping station. Perhaps each pump will even be able to compensate for the failure of its successor and its successor's successor, but perhaps not for the failure of the first 3 pumping stations in the chain of its succesors. In this case, the system fails if and only if among the sequence of pumping stations there exists a consecutive run of 3 or more failed pumping stations. This is an example of what is known as a consecutive-$k$-out-$n$:F system.

Suppose that such a system having $n$ nodes fails if and only if $k$ consecutive nodes fail, and suppose that any given node in the system works correctly with probability $p \in [0,1]$, independently of the other nodes. Then the probability that any given node fails is $q = 1 - p$. The \emph{Reliability} of the system is the probability $R(k, n; q)$ that the system does not fail. Then $R(k, n; q) = 1-P(k, n; q)$, where $P(k, n; q)$ is the probability that the sequence of $n$ nodes includes a contiguous interval of $k$ or more failed nodes. Here we are using notation as in \cite{Daus_and_Beiu_paper_2}.

As noted in \cite{Daus_and_Beiu_paper_2}, the concept of the reliability of a consecutive system was introduced to Engineering by Kontoleon in 1980 \cite{Kontoleon_paper}. In the following year Chiang and Niu \cite{Chiang_and_Nu_paper} discussed some applications of consecutive systems.
The concept has been generalised in several directions, for instance to systems deemed to have failed if and only if they include: $k$ consecutive failed components \emph{or} $f$ failed components \cite{Tung_paper}; $k$ consecutive components of which at least $r$ have failed \cite{Griffith_article}; at least $m$ non-overlapping runs of $k$ consecutive failed components \cite{Papastavridis_paper}.
In 2001 Chang, Cui and Hwang published a book on the subject \cite{Chang_Cui_and_Hwang_book}.
Four reviews are also available \cite{Chao_Fu_and_Koutras_paper}, \cite{Chang_and_Hwang_article}, \cite{Eryilmaz_paper} and \cite{Triantafyllou_paper}.

An exact formula for $R(k,n;q)$ first appeared in \cite{De_Moivre_book}. De Moivre's approach depends on deriving the generating function (see \cite{Wilf_book}) for $P(k,n;q)$. This work has been rephrased in a modern style in \cite{Uspensky_book}.
Much later a different exact formula for $R(k,n;q)$ using Markov chains was found \cite{Fu_paper}, \cite{Chao_and_Fu_paper} and \cite{Fu_and_Hu_paper}.

Both of these exact formulae for $R(k,n;q)$ are difficult to use in Engineering, where values of $k$ and $n$ may be very large. Efficient algorithms have been found (\cite{Cluzeau_Keller_and_Schneeweiss_paper} and \cite{Hwang_and_Wright_paper}) but even with modern computing resources, evaluating the formulae takes a lot of time and memory. This problem has been addressed at length in the study of upper and lower bounds for $R(k,n;q)$. Such bounds have been found which are close approximations to the exact value of $R(k,n;q)$ and which are also much easier to evaluate than the exact value of $R(k,n;q)$. In \cite{Daus_and_Beiu_paper_1}, \cite{Beiu_and_Daus_Toronto_2014} and \cite{Daus_and_Beiu_paper_2} the authors compare many such results from the 1980s to the present day.

D{\u a}u{\c s} and Beiu were the first to use De Moivre's original formula to derive upper and lower bounds for $R(k,n;q)$. As shown in \cite{Daus_and_Beiu_paper_2} this method compares very favourably with the previous upper and lower bounds, especially for small $q$, and also with the exact algorithms given in \cite{Cluzeau_Keller_and_Schneeweiss_paper} and \cite{Hwang_and_Wright_paper} (using their bounds is about 100 times faster than the algorithm from \cite{Cluzeau_Keller_and_Schneeweiss_paper}, and 1000 times faster than the algorithm from \cite{Hwang_and_Wright_paper}).

One possible extension to higher dimensions is as follows: Given $d \in \N$ and $n_1, n_2, \ldots, n_d \in \N$, suppose that some system is represented by an $n_1 \times n_2 \times \cdots \times n_d$ array of 0s and 1s. Suppose that each node in the array is 0 with probability $p \in [0, 1]$ and is 1 with probability $q = 1 - p$, independently of the other nodes. Suppose that $s_1, s_2, \ldots, s_d \in \N$ are given, with $s_r \leq n_r$ for each $r$, and suppose that the system is deemed to have failed if and only if the $d$-dimensional array includes a contiguous $s_1 \times s_2 \times \cdots \times s_d$ subarray of 1s. Define the Reliability of the system as the probability $R(s_1, \ldots, s_d, n_1, \ldots, n_d; q)$ that the system does not fail, and let $P(s_1, \ldots, s_d, n_1, \ldots, n_d; q) = 1 - R(s_1, \ldots, s_d, n_1, \ldots, n_d; q)$, so that $P(s_1, \ldots, s_d, n_1, \ldots, n_d; q)$ is the probability that the array includes a contiguous $s_1 \times s_2 \times \cdots \times s_d$ subarray of 1s.

Let us survey in chronological order some recent papers dealing with the higher dimensional case. In \cite{Salvia_and_Lasher_paper} the 2-dimensional special case $R(s,s,n,n;q)$ is treated. Upper and lower bounds are found by reducing to the 1-dimensional case. Using empirical approximations to the exact value, $R(s,s,n,n;q)$ is plotted as a function of $q$ for $n=5$ and $s = 2, 3$ and $4$, for $n=10$ and $s = 2, 4, 6$ and $8$, and for $n=50$ and $s=2, 5, 10$ and 20. The authors do not attempt to find an explicit expression for $R(s,s,n,n;q)$.
In \cite{Boehme_Kossow_and_Preuss_paper} the 2-dimensional case is considered, for $R(s_1, s_2, n_1, n_2;q)$ with $(s_1, s_2) \in \{1, n_1\} \times \{1, n_2\}$. Exact formulae are derived by comparison with the 1-dimensional case. Similar results are also given for a cylindrical version of the 2-dimensional array.
The 2-dimensional case $R(s,s,n,n;q)$ of \cite{Salvia_and_Lasher_paper} is revisited in \cite{Koutras_Papadopoulos_and_Papastavridis_paper}. Again, upper and lower bounds are given, but in the more general case where the probabilities of distinct nodes of the system failing are not necessarily equal.
For the general 2-dimensional case $R(s_1, s_2, n_1, n_2; q)$, recursive formulae for the exact reliability have been given by Yamamoto and Miyakawa \cite{Yamamoto_and_Miyakawa_paper} and also by Noguchi, Sasaki, Yanagi and Yuge \cite{Noguchi_Sasaki_Yanagi_and_Yuge_paper}. The result presented herein is distinct from their work, in that we give a closed-form formula for the exact reliability as a polynomial in $q$, for a general system, of arbitrary dimension.
The 3-dimensional case is treated in \cite{Gharib_El-Sayed_and_Nashwan_paper}, where exact formulae for the reliability are given in special cases analogous to those studied in \cite{Boehme_Kossow_and_Preuss_paper}.

In the present article we derive general closed-form exact formulae for $R(s_1, \ldots, s_d, n_1, \ldots, n_d;q)$ and for $P(s_1, \ldots, s_d, n_1, \ldots, n_d;q)$, where the dimension $d \in \N$ is arbitrary. As far as the author knows, no such formulae were known before for $d > 1$, even in the case $d=2$. Indeed, in \cite{Koutras_Papadopoulos_and_Papastavridis_paper} the authors write ``It is very difficult (probably impossible) to derive simple explicit formulas for the reliability of a general 2D-consecutive-$k$-out-$n$:F system."

\section{Results}

\begin{theorem}[Main result]
\label{thm: main result}
Let the terms
\begin{align*}
& d, \\
& s_1, \ldots, s_d, \\
& n_1, \ldots n_d, \\
& p, q, \\
& R(s_1, \ldots, s_d, n_1, \ldots, n_d; q) \\
& P(s_1, \ldots, s_d, n_1, \ldots, n_d; q)
\end{align*}
be as in the introduction, and let
\[
E = \prod_{r=1}^d \{1, \ldots, n_r - s_r + 1 \}.
\]
Then
\begin{multline}
\label{eq: formula for R}
R(s_1, \ldots, s_d, n_1, \ldots, n_d; q) = \\
1 + \sum_{J \in \mathcal{P}(E) \setminus \emptyset}
(-1)^{|J|}
\prod_{J' \in \mathcal{P}(J) \setminus \emptyset}
\exp
\left [
(-1)^{|J'|}
\ln
\left (
\frac{1}{q}
\right )
\prod_{r=1}^d
\max
\left (
0, s_r -
\left (
\max_{e \in J'} e_r - \min_{e \in J'} e_r
\right )
\right )
\right ]
\end{multline}
and
\begin{multline}
\label{eq: formula for P}
P(s_1, \ldots, s_d, n_1, \ldots, n_d; q) = \\
- \sum_{J \in \mathcal{P}(E) \setminus \emptyset}
(-1)^{|J|}
\prod_{J' \in \mathcal{P}(J) \setminus \emptyset}
\exp
\left [
(-1)^{|J'|}
\ln
\left (
\frac{1}{q}
\right )
\prod_{r=1}^d
\max
\left (
0, s_r -
\left (
\max_{e \in J'} e_r - \min_{e \in J'} e_r
\right )
\right )
\right ]
\end{multline}
where $\mathcal{P}(A)$ denotes the power set of the set $A$, that is, the set of subsets of $A$, and $\emptyset$ denotes the empty set.
\end{theorem}

Theorem \ref{thm: main result} implies the following combinatorial result:
\begin{corollary}
\label{cor: combinatorial formula}
With all terms as in Theorem \ref{thm: main result}, the number $a(s_1, \ldots, s_d, n_1, \ldots, n_d)$ of failed systems among the $2^{\prod_{r=1}^d n_r}$ possible systems is
\begin{multline*}
a(s_1, \ldots, s_d, n_1, \ldots, n_d) = \\
-2^{\prod_{r=1}^d n_r}
\sum_{J \in \mathcal{P}(E) \setminus \emptyset}
(-1)^{|J|}
\prod_{J' \in \mathcal{P}(J) \setminus \emptyset}
\exp
\left [
(-1)^{|J'|}
\ln
\left (
2
\right )
\prod_{r=1}^d
\max
\left (
0, s_r -
\left (
\max_{e \in J'} e_r - \min_{e \in J'} e_r
\right )
\right )
\right ].
\end{multline*}
\end{corollary}
\begin{proof}
In case $p = q = \frac{1}{2}$, all possible systems occur with equal probability, so the probability measure on the power set of the set of all systems becomes the counting measure, divided by a factor of $2^{\prod_{r=1}^d n_r}$.
\end{proof}

Setting $d=1$, $n_1 = n$ and $s_1=2, 3$ and $4$ in Corollary \ref{cor: combinatorial formula} yields the sequences \cite{Sloane_Kennedy_Weisstein_sequence}, \cite{Weisstein_sequence_1} and \cite{Weisstein_sequence_2}, respectively.
Setting $d=2$, $n_1 = n_2 = n$ and $s_1 = s_2 = 2$ in Corollary \ref{cor: combinatorial formula} yields the sequence \cite{Cowell_sequence}, whose complementary sequence is \cite{Hardin_sequence}.

\begin{proof}[Proof of Theorem \ref{thm: main result}]
Call a system $e$ an ``elementary failure case" if and only if it includes a contiguous $s_1 \times \ldots \times s_d$ subarray of 1s, and its other elements are all 0.
We identify the set of all elementary failure cases $e$ with the set $E$, according to the rule
\[
e \leftrightarrow (e_1, \ldots, e_d)
\]
if and only if the $(i_1, \ldots, i_d)$-th element of $e$ is 1 precisely when, for all $r \in \{1, \ldots, d\}$,
\[
e_r \leq i_r \leq e_r + s_r - 1.
\]
For example, in the two-dimensional case, we identify an elementary failure case with the indices $(e_1,e_2)$ of the top left corner of its contiguous subrectangle of 1s.
Denote by $G$ the simple acyclic directed graph whose vertices are all the possible systems, and in which system $b$ is a direct successor of system $a$ if and only if $b$ is obtained from $a$ by changing a single 0 element to 1.
Then the failed systems $f$ in $G$ are precisely those vertices having some elementary failure case $e \in E$ as an ancestor (including the possibility that $f = e$.)
Therefore if $F$ denotes the set of all failed systems $f$ (in terms of probability, $F$ is the event that the system fails,) then
\[
F = \bigcup_{e \in E} \text{the set of descendents of $e$},
\]
where by ``$b$ is a descendent of $a$" we mean to include the possibility that $b = a$.
Then
\begin{align*}
P(s_1, \ldots, s_d, n_1, \ldots, n_d; q)
&= P(F) \\
&= P 
\left (
\bigcup_{e \in E}
\text{the set of descendents of $e$}
\right ) \\
&= \sum_{J \in \mathcal{P}(E) \setminus \emptyset}
(-1)^{|J| + 1}
P 
\left (
\bigcap_{e \in J}
\text{the set of descendents of $e$}
\right ),
\end{align*}
by the Inclusion-Exclusion Principle for the probability measure on the power set of the set of all possible systems.
But
\[
\bigcap_{e \in J} \text{the set of descendents of $e$} = \text{the set of descendents of $\mathrm{hcd}(J)$},
\]
where by $\mathrm{hcd}(K)$ we mean the highest common descendent of any given set $K$ of vertices in $G$. That is, in the partial order on the vertices of $G$ defined by $a \lesssim b$ precisely when $a$ is an ancestor of $b$, $\mathrm{hcd}(K)$ is the least upper bound of $K$.
Moreover,
\[
\mathrm{hcd}(J) = \bigcup_{e \in J} e,
\]
where by the union of several $n_1 \times \cdots \times n_d$ systems we mean the system whose generic element is 1 if and only if the corresponding element in at least one of those systems is also 1.
Thus
\begin{align}
\label{eq: star}
P(s_1, \ldots, s_d, n_1, \ldots, n_d; q)
&=
\sum_{J \in \mathcal{P}(E) \setminus \emptyset}
(-1)^{|J|+1}
P
\left (
\bigcap_{e \in J}
\text{the set of descendents of $e$}
\right ) \\ \nonumber
&=
\sum_{J \in \mathcal{P}(E) \setminus \emptyset}
(-1)^{|J|+1}
P
\left (
\text{the set of descendents of $\bigcup_{e \in J} e$}
\right ).
\end{align}
However, for any vertex $a$ of $G$,
\begin{align*}
P(\text{the set of descendents of $a$})
&=
\sum_{\text{$b$ is a descendent of $a$}} P(\{b\}) \\
&=
\sum_{\text{$b$ is a descendent of $a$}} p^{n_1 n_2 \ldots n_d - k} q^k,
\end{align*}
where $k = k(b)$ is the number of failed nodes in the system $b$.
Gathering like powers of $q$ in the above sum, we have
\begin{align*}
P(\text{the set of descendents of $a$})
&=
\sum_{k=k(a)}^{n_1 n_2 \ldots n_d}
{n_1 n_2 \ldots n_d - k(a) \choose k-k(a)} p^{n_1 n_2 \ldots n_d - k} q^k \\
&= q^{k(a)}
\sum_{k=k(a)}^{n_1 n_2 \ldots n_d} {n_1 n_2 \ldots n_d - k(a) \choose k-k(a)} p^{n_1 n_2 \ldots n_d - k} q^{k-k(a)} \\
&= q^{k(a)}
\sum_{j=0}^{n_1 n_2 \ldots n_d - k(a)} {n_1 n_2 \ldots n_d - k(a) \choose j} p^{n_1 n_2 \ldots n_d - k(a) - j} q^j \\
&=
q^{k(a)} (p+q)^{n_1 n_2 \ldots n_d - k(a)} \\
&= q^{k(a)}.
\end{align*}
That is, for any vertex $a$ of $G$,
\[
P(\text{the set of descendents of $a$}) = q^{k(a)},
\]
where $k(a)$ denotes the number of failed nodes in the system $a$.
Therefore from \eqref{eq: star} we obtain
\begin{align}
\label{eq: star star}
P(s_1, \ldots s_d, n_1, \ldots n_d; q)
&=
\sum_{J \in \mathcal{P}(E) \setminus \emptyset}
(-1)^{|J|+1}
P
\left (
\text{the set of descendents of $\bigcup_{e \in J} e$}
\right ) \\ \nonumber
&=
\sum_{J \in \mathcal{P}(E) \setminus \emptyset}
(-1)^{|J|+1}
q^{k
\left (
\bigcup_{e \in J} e
\right )
}
.
\end{align}
The next task is to compute the number $k \left ( \bigcup_{e \in J} e \right )$ of 1s in the system $\bigcup_{e \in J} e$, where $J$ is some nonempty subset of the set $E$ of elementary failure cases $e$.
We use the Inclusion-Exclusion Principle again, this time for the counting measure on the power set of the set of elements of a system, to obtain
\begin{equation}
\label{eq: star star star}
k
\left (
\bigcup_{e \in J} e
\right )
=
\sum_{J' \in \mathcal{P}(J) \setminus \emptyset}
(-1)^{|J'|+1}
k
\left (
\bigcap_{e \in J'} e
\right ).
\end{equation}
Each elementary failure case consists entirely of 0s except for a contiguous $s_1 \times \ldots \times s_d$ subarray of 1s.
Therefore $\bigcap_{e \in J'} e$ also consists entirely of 0s except for a (possibly empty) contiguous subarray of dimensions $t_1 \times \cdots \times t_d$, where for each $r \in \{1, \ldots d\}$,
\begin{align*}
t_r
&= \max
\left (
0,
\left (
\min_{e \in J'}
\left (
e_r + s_r - 1
\right )
- \max_{e \in J'} e_r
+ 1
\right )
\right ) \\
&=
\max
\left (
0, s_r -
\left (
\max_{e \in J'} e_r - \min_{e \in J'} e_r
\right )
\right ).
\end{align*}
The volume of that contiguous $t_1 \times \cdots \times t_d$ subarray is equal to $k \left ( \bigcap_{e \in J'} e \right )$, that is,
\begin{align}
\label{eq: star star star star star}
k
\left (
\bigcap_{e \in J'} e
\right )
&=
\prod_{r=1}^d t_r \\ \nonumber
&=
\prod_{r=1}^d
\max
\left (
0, s_r -
\left (
\max_{e \in J'} e_r - \min_{e \in J'} e_r
\right )
\right ) .
\end{align}
Considering \eqref{eq: star star}, \eqref{eq: star star star} and \eqref{eq: star star star star star}, we have
\begin{multline*}
P(s_1, \ldots, s_d, n_1, \ldots, n_d; q)
=
\sum_{J \in \mathcal{P}(E) \setminus \emptyset}
(-1)^{|J|+1}
q^{k
\left (
\bigcup_{e \in J} e
\right )
} \\
=
\sum_{J \in \mathcal{P}(E) \setminus \emptyset}
(-1)^{|J|+1}
\exp
\left [
\ln (q)
\sum_{J' \in \mathcal{P}(J) \setminus \emptyset}
(-1)^{|J'|+1}
k
\left (
\bigcap_{e \in J'} e
\right )
\right ] \\
=
\sum_{J \in \mathcal{P}(E) \setminus \emptyset}
(-1)^{|J|+1}
\exp
\left [
\ln (q)
\sum_{J' \in \mathcal{P}(J) \setminus \emptyset}
(-1)^{|J'|+1}
\prod_{r=1}^d
\max
\left (
0, s_r -
\left (
\max_{e \in J'} e_r - \min_{e \in J'} e_r
\right )
\right )
\right ] \\
=
-\sum_{J \in \mathcal{P}(E) \setminus \emptyset}
(-1)^{|J|}
\prod_{J' \in \mathcal{P}(J) \setminus \emptyset}
\exp
\left [
(-1)^{|J'|}
\ln
\left (
\frac{1}{q}
\right )
\prod_{r=1}^d
\max
\left (
0, s_r -
\left (
\max_{e \in J'} e_r - \min_{e \in J'} e_r
\right )
\right )
\right ],
\end{multline*}
which proves equation \eqref{eq: formula for P}. Equation \eqref{eq: formula for R} follows immediately.
\end{proof}


As an example, we compute the reliabilities of 3 consecutive systems, having dimensions 1, 2 and 3:

\begin{align*}
R(1,2;q)	&= 1 - 2 q + q^2, \\
R(1,2,2,3;q)	&= 1-4 q^2+2 q^3+4 q^4-4 q^5+q^6, \\
R(1,2,3,2,3,4;q)	&= 1-8 q^6+4 q^8+4 q^9+4 q^{10}-8 q^{11}+18 q^{12}-16 q^{14} -16 q^{15} \\
& \qquad -12 q^{16}+40 q^{17}+4 q^{18}-8 q^{19}-8 q^{20}-12 q^{21}+20 q^{22}-8 q^{23}+q^{24}.
\end{align*}

See also Figure 1.

\vspace{1cm}

\newpage

\begin{center}
Figure 1 \\

\nopagebreak

\vspace{5mm}
\includegraphics[width=.9\columnwidth]{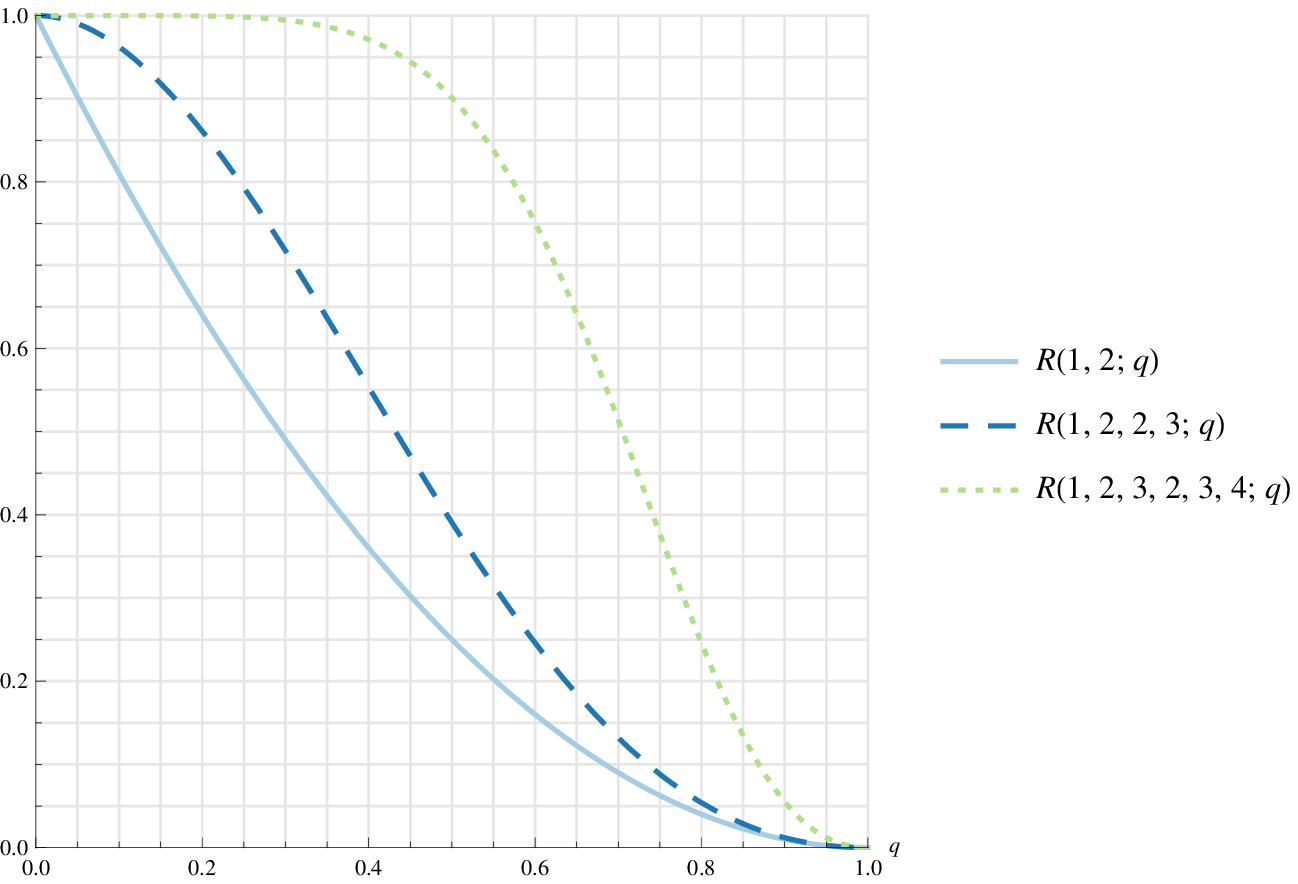}
\end{center}



\section{Conclusion}

We have provided a novel formula for the reliability of a general $d$-dimensional consecutive-$k$-out-of-$n$:F system, as an exact polynomial in $q$. We believe ours to be the first general exact closed-form formula published for the reliability in $d$ dimensions. This answers an open question in Engineering.

\section{Conflict of Interests}

The author declares that there is no conflict of interests regarding the publication of this manuscript.

\end{document}